\documentclass[a4paper,12pt]{article}
\usepackage[numbers]{natbib}
\usepackage{amsthm}
\usepackage{amssymb} 
\usepackage{amsmath}
\usepackage{amsfonts}
\usepackage{verbatim}
\usepackage{centernot} 
\usepackage{stmaryrd} 
\usepackage{mathrsfs} 
\usepackage{paralist} 
\usepackage{tikz}
\usetikzlibrary{arrows,shapes}

\newlength{\ancho}
\textwidth=\ancho
\newcommand{\rel}[1]{\mathrel{\mathcal{R}\left(#1\right)}}
\newcommand{\sem}[1]{\llbracket #1 \rrbracket}
\newcommand{\modal}[1]{\langle #1 \rangle}
\newcommand{\qmodalgq}[2]{[#1]_{{\geq} #2}}
\newcommand{\bisim}{\sim}
\newcommand{\sbisim}{\bisim_{\rm s}}
\newcommand{\ebisim}{\bisim_{\rm e}}
\newcommand{\tbisim}{\bisim_{\rm t}}
\newcommand{\nlmp}[1]{\mathbf{#1}}

\newcommand{\iso}{\cong}
\renewcommand{\emptyset}{\varnothing}
\newcommand{\vacia}{\epsilon}
\newcommand{\et}{\mathrel{\&}}
\newcommand{\dia}{\lozenge}

\newcommand{\calS}{\Sigma}
\newcommand{\calU}{\Xi}

\newcommand{\calL}{\mathscr{L}}

\newcommand{\Pow}{\mathsf{Pow}}

\newcommand{\Q}{\mathbb{Q}}
\newcommand{\N}{\mathbb{N}}
\newcommand{\B}{\mathcal{B}}
\newcommand{\E}{\mathbf{E}}
\newcommand{\bB}{\mathbf{B}}
\newcommand{\bA}{\mathbf{A}}

\renewcommand{\>}{\rangle}
\newcommand{\pre}[2]{\langle #1\rangle #2}
\newcommand{\keywords}[1]{{\renewcommand{\thefootnote}{\relax}\footnotetext{\emph{Keywords:}
    #1}}}

\newcommand{\ent}{\Rightarrow}

\newcommand{\sii}{\Leftrightarrow}
\renewcommand{\phi}{\varphi}

\renewcommand{\th}{\theta}

\newcommand{\Lda}{\Lambda}
\newcommand{\lda}{\lambda}
\newcommand{\sig}{\ensuremath{\sigma}}

\newcommand{\Sig}{\ensuremath{\boldsymbol\Sigma}}

\newcommand{\dist}{\Delta}

\newcommand{\Tree}{{\mathrm{Tr}_\N}}
\newcommand{\NWO}{\mathit{NWO}}
\newcommand{\Suc}{{\N^*}}

\newcommand{\ap}{\mathord{^\smallfrown}}
\newcommand{\pred}{\prec}
\newcommand{\bpred}{\mathrel{\overline{\prec}}}

\newcommand{\npred}{\mathrel{\nprec}}
\newcommand{\disr}{\triangleright}

\newcommand{\om}{\ensuremath{\omega}}

\newcommand{\sbt}{{\,\begin{picture}(-1,1)(-1,-1)\circle*{2}\end{picture}\ }}
\newcommand{\Tfont}[1]{\mathsf{#1}}
\newtheorem{theorem}{Theorem}
\newtheorem{lemma}[theorem]{Lemma}
\newtheorem{prop}[theorem]{Proposition}
\newtheorem{corollary}[theorem]{Corollary}

\theoremstyle{definition}
\newtheorem{definition}[theorem]{Definition}
\newtheorem{example}[theorem]{Example}
\theoremstyle{remark}
\newtheorem*{ack}{Acknowlegdements}
\usepackage[matrix,arrow,curve]{xy}

\begin{document}
\title{Bisimilarity is not Borel}
\author{Pedro S\'{a}nchez Terraf\thanks{The author was partially
    supported by CONICET, ANPCyT project PICT 2012-1823, SeCyT-UNC project 05/B284, and EU 7FP grant agreement 295261 (MEALS). Part of
    this work was presented at Dagstuhl Seminar 12411 on \emph{Coalgebraic
      Logics.}}}
\maketitle
\keywords{measurable labelled transition system, non-deterministic labelled Markov process,
modal logic, Borel hierarchy. \emph{MSC 2010:} 03B70; 03E15; 28A05. \emph{ACM class:} F.4.1; F.1.2.}
\begin{abstract}
We prove that the relation of bisimilarity between 
countable labelled transition systems is $\Sig_1^1$-complete (hence
not Borel), by reducing the
set of  non-wellorders over the natural numbers continuously to it. 

This has an impact on the theory of probabilistic and nondeterministic processes 
over uncountable spaces,  since logical characterizations of 
bisimilarity (as, for instance, those based on the unique structure
theorem for analytic spaces) require a  
countable logic whose formulas have measurable semantics. Our reduction shows 
that such a logic does not exist in the case of image-infinite
processes.
\end{abstract}

\section{Introduction}
Markov decision processes over continuous state spaces are an
appropriate framework to study and formalize systems that involve
continuously valued variables such as those arising in physics, biology,
and economics; and where some of those variables are
known only in a probabilistic way.

In this direction, \emph{labelled Markov processes (LMP)} were developed in
\cite{Desharnais,DEP} by Desharnais et alter. A LMP has  a labelled set of \emph{actions}
that encode interaction with the environment; thus LMP are a
reactive model in which there are different transition
subprobabilities for each action. In this model, uncertainty
is (only) considered to be probabilistic; therefore, LMP can be
regarded as generalization of deterministic processes. Intimately
related to LMP are \emph{stochastic relations}
\cite{doberkat2007stochastic}; LMP are exactly the
stochastic Kripke \emph{frames}. It will be convenient, in search for motivation, to draw a
panorama of the problems addressed by the present work in the context
of LMP.

There are several ways to face the problem of `equivalence of
behavior' or \emph{bisimilarity} for
LMP. One possible but rather arbitrary way to classify them is as
\emph{relational}, \emph{coalgebraic} and \emph{logical}, and at the
core of the matter that concerns us are the connections  between these
different roads.  Each one has its own advantages and sharp edges.

In the relational way, one uses some natural
generalizations to continuous spaces of  \emph{probabilistic  bisimilarity} as defined by
Larsen and Skou \cite{LarsenSkou}. Hence bisimulations and the
relation of bisimilarity are simply binary relations on the base
space. This way to define bisimulations is the standard one for Kripke
models and will play a central role in this paper. Another, perhaps
more important observation is that this relational view is
\emph{internal}: it only works inside a single process. When one wants to
compare states in two different processes \emph{relationally}, it's necessary  to construct
their direct sum first. Since bisimulations
in this sense have no structure (they're plain sets of pairs) facts as
the transitivity of the relation of bisimilarity are
straightforward.

In the coalgebraic view, 
\cite{Rutten00,doberkat2009stochastic} there are two options for
describing behavior of processes: coalgebraic 
bisimilarity is the existence of a span of zigzag morphisms and the
concept of \emph{behavioral equivalence}, which is given by a cospan of
morphisms. Here the problem of transitivity is not trivial. It can be
shown by functorial manipulations that behavioral equivalence is transitive. But in the case of spans, this is a deep
theorem. There are two approaches to this problem. One is to apply a
technical lemma by Edalat \cite{Edalat} which ensures the existence of
semipullbacks 
in the category of LMP over analytic spaces; that is, every cospan can
be completed to a commutative square. The second one allows the
construction to be carried out without leaving the realm of Polish
spaces, in case one starts with a cospan of LMP over Polish spaces. This
was achieved by Doberkat \cite{DBLP:journals/siamcomp/Doberkat05} by using a selection
argument. Coalgebraic bisimilarity is a refinement of the relational
view, since the mediating process (the ``vertex'' of the span) can be
regarded as binary relation carrying a LMP structure. Note, however,
that the passage to the relational view is not completely smooth, since
\cite[Example 1]{Pedro20111048} shows that direct sum is
not always compatible with coalgebraic bisimilarity. 

Finally, in the logical view one fixes 
a variant of Hennessy-Milner  modal
logic. Formulas of the modal logic can be regarded as tests
on a process and we call
two states \emph{logically equivalent} if they satisfy exactly the same
formulas. We don't put logical equivalence under the ``relational
approach'' since the main ingredients are the logical formulas;
moreover, logic is \emph{external}. There
is an alternative but equivalent view, called ``event
  bisimulation'' \cite{coco} where emphasis is put on families of measurable
subsets of the state space. The connection between these two is that
the \emph{semantics (or validity sets) of Hennessy-Milner formulas are
measurable sets} for each interpretation (process). This is a recurrent
theme that will be exploited in 
our results. There is a generalization of this logical/event approach
through the investigation of the lifting of countably generated
equivalence relations to the space of probabilities \cite{Doberkat2007638}.

As in older and related areas as Kripke frames and models, one of the main
problems about LMP is to make a link between the logical and the
other approaches, specially to relational and coalgebraic
bisimilarity. For the former, the problem of \emph{logical
  characterization of bisimilarity} on a class of processes consist in proving that
(relational) bisimilarity is the same as logical equivalence for each
process in the class.  One of the cleanest way to do this is to apply the
\emph{unique structure theorem}  for   analytic spaces, as it
was done in the work of Danos et 
al.~\cite{coco}. This works for the class of LMP with analytic state
spaces, but it can be shown that both notions actually differ in the
class of LMP over general measurable spaces, and that coalgebraic
bisimilarity is not even transitive \cite{Pedro20111048} (note that,
trivially, logical equivalence is).

Our interest lies in models that include both probability and internal
nondeterminism. These arise
naturally, e.g., by abstraction or underspecification of LMP. In the discrete case, the
class of 
\emph{probabilistic automata} is an example. Over uncountable state
spaces, the common generalization of LMP and probabilistic automata is
given by
\emph{nondeterministic labelled Markov processes (NLMP)}
\cite{DWTC09:qest,Wolovick}. NLMP allow, for each state $s$  and each
action $a$, a (possibly
infinite) set of probabilistic behaviors $\Tfont{T}_a(s)$. Deterministic NLMP (i.e., for which
$\Tfont{T}_a(s)$ is a singleton for each $a,s$) are essentially the same as
LMP. (It's worth to mention a very different
approach to underspecification, using 
super-additive functions, proposed in
\cite{DBLP:journals/iandc/DesharnaisLT11}.)

In \cite{DWTC09:qest,D'Argenio:2012:BNL:2139682.2139685} the problem
of defining appropriate notions of bisimulation and finding logical
characterizations for bisimulation of NLMP was addressed. 
 It turns out
that there are three
different notions of bisimilarity for these processes: traditional, state-based and
event-based.  The first two are ``relational'' in nature, and for
deterministic NLMP they
collapse to state bisimilarity. (It should  be noted that a
neat coalgebraic presentation of NLMP is still missing, so we won't be
concerned with the coalgebraic approach in this work.) Event bisimilarity is
analogous to the concept for LMP bearing the same name and hence it is
characterized by a logic, though  one far more complex than
Hennessy-Milner's. It has two levels and uncountably many formulas in
the general case. 

Indeed, we witness a similar  phenomenon as in
labelled transition systems (LTS) vs.\ modal logic, where some logical
operators (e.g. conjunction) must have arity at 
least as the branching of the process. In our case we can do by just
using countable conjunctions, in  spite the  sets $\Tfont{T}_a(s)$ may have
the cardinality of the continuum. In any case, by using infinitary
operators one is lead to an uncountable logic. 

Nevertheless, the arguments used in \cite{coco}  
can be generalized to encompass image-finite NLMP (i.e., having all  the  sets $\Tfont{T}_a(s)$
finite) over analytic state spaces. Actually, a proof strategy can be found in
\cite{DWTC09:qest,Celayes}: every countable `measurable' logic $\calL$ 
satisfying certain local restrictions must characterize
bisimilarity. (Here we call a logic \emph{measurable} if the validity set
of each formula is a measurable subset of the state space.) Both
countability and measurability requirements are necessary for the
proof to work. As far as we know, all approaches to the problem of
logical characterization of bisimilarity need a logic satisfying these
two properties.

Several examples in
\cite{DWTC09:qest,D'Argenio:2012:BNL:2139682.2139685} show that in the
case of uncountable branching, our logic does not characterize neither
of the relational bisimilarities considered, and that they are
different. So the remaining case is the one of countably infinite
branching. In this paper we show that the relation of bisimilarity is not Borel
in an appropriate countably branching process $\nlmp F$ having a
Polish state space, and therefore 
we prove that there is no 
countable measurable logic that characterizes bisimilarity in any
class of NLMP containing $\nlmp F$, thus banishing hope for using the
current techniques for proving logical characterization. We will do this in two steps:
first we'll see that the relation of bisimilarity on the space $\Tree$ of all
trees on $\N$  is a non Borel set (actually, it is
analytic-complete). We will then use $\Tree$ to build  the state
space of $\nlmp F$, providing a NLMP structure.

In the next section
we review some of the known results on NLMP, describing  the
available notions of bisimulation. Most calculations in the paper will
be carried on a NLMP where the sets $\Tfont{T}_a(s)$ consist entirely of point
masses (Dirac's deltas). This simpler model has another, more
practical, presentation which is
essentially a LTS over a measurable space (Section \ref{sec:measurable-lts}). In
Section 3 we use some machinery of sequence spaces and unwinding of
labelled transition systems to assess the complexity of the relation
of bisimilarity. The final section contains some concluding remarks.

\section{Review of NLMP}
\subsection{Basic definitions}
All of the material of this section appears in
\cite{D'Argenio:2012:BNL:2139682.2139685}. Let $(S,\calS)$ be a
measurable space. The set $\dist(S)$  of probability measures on
$(S,\calS)$ has a natural \sig-algebra $\dist(\calS) =
\sig(\{\dist^{\geq q}(Q) : q\in \Q, Q\in \calS\})$, where $\dist^{\geq
  q}(Q)=\{\mu\in\dist(S) : \mu(Q)\geq q\}$. This is the least
\sig-algebra making evaluation $\mu\mapsto\mu(Q)$ measurable.

Recall that a Markov kernel
on $(S,\calS)$ is a measurable 
map $\Tfont{T}:(S,\calS)\to (\dist(S),\dist(\calS))$. The following definitions generalize
this concept by enlarging the codomain to the family
of  all measurable sets of
probability measures ($\dist(\calS)$), and constructing a \sig-algebra
for this family in order to be able to say that $\Tfont{T}$ is
measurable. In this section, the $\Lda$ appearing in
Definition~\ref{def:hit-salg} will be $\dist(\calS)$; in
Section~\ref{sec:measurable-lts} the choice of $\Lda$ will be different. 
\begin{definition}\label{def:hit-salg}%
  Let $X$ be a set and let $\Lda$ be some family of sets. $H(\Lda)$ is the least
  $\sigma$-algebra on $\Lda$ containing all sets
  $H_{\xi} \doteq
    \{ \zeta\in\Lda :
                   \zeta\cap\xi \neq\emptyset \}$
  with $\xi\in\Lda$.
\end{definition}
\begin{definition}\label{def:NLMP}%
  A \emph{nondeterministic labelled Markov process} (NLMP) is
  a tuple $\nlmp S = (S,\Sigma, \{\Tfont{T}_a : a\in L \})$ where $\Sigma$ is a
  $\sigma$-algebra on the set of states $S$, and for each label $a\in
  L$, $\Tfont{T}_a : (S,\Sigma) \to (\Delta(\Sigma),H(\Delta(\Sigma)))$ is
  measurable. We call $\nlmp S$ \emph{image-finite (image-countable)} if all
  the sets $\Tfont{T}_a(s)$ are finite (countable). NLMP that are not
  image-finite (image-countable) are called \emph{image-infinite} (\emph{image-uncountable}). 
\end{definition}
The motivation for the previous definitions is that we want the event
``there exists a probabilistic behavior from $s$ such that
\dots'' to be measurable; we want this in order to be able to calculate the probability of
such event (cf. the semantics of the logic below).

Some notation concerning
binary relations will be needed to define bisimulations. Let
$R$ a binary relation over $S$. 
A set $Q$ is \emph{$R$-closed} if $Q\ni x\mathrel{R} y$ implies $y\in
Q$. $\calS(R)$ is the \sig-algebra of $R$-closed sets in $\calS$. If
$\mu, \mu'$ are measures defined on $\calS$, we write $\mu \mathrel{R}
\mu'$ if they coincide in $\calS(R)$. Lastly, 
let $\calU$ be a subset of $\Pow(S)$, the powerset of $S$. The relation $\rel{\calU}$
is given by:
\[(s,t)\in{\rel{\calU}} \quad \iff \quad \forall Q\in \calU: s\in Q \sii
t\in Q.\]
\begin{definition}
  \begin{enumerate}
  \item \label{def:nlmp_eb}%
    An \emph{event bisimulation} on a NLMP $(S,\Sigma,\{\Tfont{T}_a : a\in
    L\})$ is a sub-$\sigma$-algebra $\Lambda$ of $\Sigma$ such that
    $\Tfont{T}_a:(S,\Lambda)\to(\dist(\Sigma),H(\dist(\Lambda)))$ is measurable
    for each $a\in L$. We also say that a
    relation $R$ is an event bisimulation if there is an event
    bisimulation $\calU$ such that $R=\mathcal{R}(\calU)$.
  \item \label{def:nlmp_sb}%
    A relation $R\subseteq S\times S$ is a \emph{state bisimulation}
    if it is symmetric and for all $a\in L$, $s\mathrel{R}t$ implies
    $\forall\xi\in\Delta(\Sigma(R)) :
    \Tfont{T}_a(s)\cap\xi\neq\emptyset \iff \Tfont{T}_a(t)\cap\xi\neq\emptyset$.
  \item \label{def:nlmp_ls_sb}%
    A relation $R$ is a \emph{traditional bisimulation} if it is
    symmetric and for all $a\in L$, $s\mathrel{R}t$ implies that
    for all 
    $\mu\in \Tfont{T}_a(s)$ there exists $\mu'\in \Tfont{T}_a(t)$ such that
    $\mu\mathrel{R}\mu'$.
  \end{enumerate}
  We say that $s,t\in S$ are traditionally (resp.\ state-, event-) \emph{bisimilar}, denoted
  by $s \tbisim t$ ($s \sbisim t$, $s \ebisim t)$, if there is a
  traditional (state, event) bisimulation $R$ such that $s\mathrel{R}t$.
\end{definition}
We want to stress the fact that each notion of bisimulation/bisimilarity is
defined relative to a particular NLMP. 
Event bisimulation is a
straightforward generalization of the same concept for LMP, and it is
the one that is most ``compatible'' with the measurable structure of the
base space. For LMP, it is equivalent to the existence of a cospan of
morphisms; see \cite{coco}. Traditional bisimulation is in a sense the
most faithful generalization of both probabilistic bisimulation by
Larsen and Skou and the standard notion of bisimulation for non
deterministic processes, e.g., LTS. Finally, state bisimilarity is a
good trade-off between the other two, since it is generally finer than event
bisimilarity but it is a little more respectful to the measurable
structure than the traditional version; in this case, we only ask
transition sets to \emph{hit} the same $R$-closed sets of measures, and this is a
weaker requirement than that of traditional bisimulations. We refer the reader to 
\cite{D'Argenio:2012:BNL:2139682.2139685} for further discussion on
these notions of bisimilarity.%
\footnote{In later works different (and we hope better) names for these
  notions are used: \emph{state} and \emph{hit} bisimulations, instead
  of traditional and state bisimulations, respectively.}%


To close this section, we introduce the logic $\calL$. Consider two kinds
of formulas: one that is interpreted on states, and another that is
interpreted on measures. 
\begin{eqnarray*}
  \phi & \ \equiv \ &\textstyle
    \top \ \mid \ \phi_1\land\phi_2 \ \mid \ \modal{a}\psi \\
  \psi & \ \equiv \ &\textstyle
    \bigvee_{i\in I} \psi_i \ \mid \ \neg\psi \ \mid \ \qmodalgq{\phi}{q}
\end{eqnarray*}
where $a\in L$, $I$ is a countable index set, and
$q\in\Q\cap{[0,1]}$. We denote by $\calL$ the set of all formulas
generated by the first production. The semantics of the logic is given
relative to a NLMP  $(S,\Sigma, \{\Tfont{T}_a : a\in L \})$. 
\begin{align*}
  & \sem{\top} = S
  \quad  & \
  & \sem{\textstyle\bigvee_{i\in I} \psi_i} =  \textstyle\bigcup_i \sem{\psi_i}  \\
  & \sem{\phi_1\land\phi_2} =  \sem{\phi_1}\cap\sem{\phi_2}
  \quad  & \
  & \sem{\neg\psi} = \dist(S)\setminus\sem{\psi}  \\
  & \sem{\modal{a}\psi} = \Tfont{T}_a^{-1}(H_{\sem{\psi}})
  \quad  & \
  & \sem{\qmodalgq{\phi}{q}} = \Delta^{\geq q}(\sem{\phi})
\end{align*}
where $\sem{\chi}$ denotes the validity set (or extension) of $\chi$. We may explain a
little bit the last to clauses. Expanding the one with the modality,
we obtain
\[\sem{\modal{a}\psi} = \{s\in S : \exists (\mu \in \Tfont{T}_a(s) \cap
\sem{\psi})\},\]
so that this is the set of all states such that under action $a$,
there is a transition to a measure $\mu$ satisfying $\psi$. The test
on measures expands to
\[\sem{\qmodalgq{\phi}{q}} =\{\mu \in \dist(S) : \mu(\sem{\phi})\geq
q\},\]
so a measure $\mu$ satisfies $\qmodalgq{\phi}{q}$ if it assigns
probability greater or equal than $q$ to the set of states satisfying
$\phi$.
\begin{example}
  In case the NLMP $\nlmp S = (S,\Sigma, \{\Tfont{T}_a : a\in L \})$ is deterministic,
  meaning $\Tfont{T}_a(s)=\{k_a(s)\}$ is a singleton for each $a$ and $s$, it
  can be proved that $k_a$ are Markov kernels and then $\nlmp S$ is
  essentially a LMP. We may easily codify the logic used for these
  processes \cite{Desharnais} in
  $\calL$ as synthetic sugar; we should only worry about the modality
  $\<a\>_q$:
  \[\sem{\<a\>_q\phi} =  \{s \in S : k_a(s)(\sem{\phi}) \geq q\} =
  \sem{\modal{a}{\qmodalgq{\phi}{q}}}.\] 
\end{example}

It can be proved
by induction that all the sets $\sem{\chi}$ are measurable in the
respective spaces by using $H(\dist(\Sigma))$-measurability. Let $\sem{\calL}\doteq\{\sem{\phi} : \phi \in\calL\}$.
\begin{theorem}\label{thm:charaterization_event_bisimulation}%
  The logic $\calL$ completely characterizes event bisimulation.
  In other words, ${\rel{\sem\calL}}={\ebisim}$.
\end{theorem}
\begin{theorem}\label{theorem:logic_is_sound}
  ${\tbisim} \subseteq {\sbisim} \subseteq {\ebisim} = {\rel{\sem\calL}}$.
\end{theorem}
%
%
The following Lemma summarizes one strategy to prove completeness of a
logic for traditional bisimilarity. We begin by recalling some
definitions; the reader may review them in Kechris \cite{Kechris} or
Srivastava \cite{books/daglib/0029964}. A topological space $(S,\mathcal{T})$ is \emph{Polish} if it is separable and completely
metrizable and an \emph{analytic}  space is the
image of a continuous map between Polish spaces. We call
$\B(\mathcal{T})\doteq\sig(\mathcal{T})$ (or $\B(S)$ if the topology
is understood from the context) the Borel \sig-algebra of the space
$(S,\mathcal{T})$. We say that
a measurable space is an \emph{analytic Borel}  if it is isomorphic to
some $(S,\B(\mathcal{T}))$ with $(S,\mathcal{T})$ analytic.
\begin{lemma}\label{lm:anylogic_to_bisim}
  Let $(S,\Sigma,T)$ be a NLMP with $(S,\Sigma)$ being an analytic
  Borel space. Let
  $\mathfrak{L}$ be a logic such that
  \begin{inparaenum}[\it (i)]
  \item%
    $\mathfrak{L}$ contains operators $\top$ and $\land$ with the
    usual semantics;
  \item%
    for every formula $\phi\in\mathfrak{L}$,
    $\sem{\phi}\in\Sigma$;
  \item%
    the set of all formulas in $\mathfrak{L}$ is countable; and
  \item%
    for every $s\rel{\mathfrak{L}}t$ and every $\mu\in \Tfont{T}_a(s)$
    there exists $\mu'\in \Tfont{T}_a(t)$ such that $\forall
    \phi\in\mathfrak{L}, \mu(\sem{\phi}) = \mu'(\sem{\phi})$.
  \end{inparaenum}
  Then, two logically equivalent states $s,t$ are traditionally
  bisimilar.
\end{lemma}
The proof of this lemma is based on the \emph{unique structure
  theorem} for analytic Borel spaces $(S,\calS)$: every countably
generated sub-$\sig$-algebra of $\calS$ that separates points must
already be $\calS$.

By using this Lemma we were able to prove that a countable fragment of
$\calL$ is complete for traditional bisimilarity over image-finite
NLMP on analytic state-spaces. The next step would be to prove a
similar result for 
image-countable processes, and the safest way to test this is in a
more ``discrete'' setting. In the next section we consider a
restricted class of processes.
\subsection{Measurable LTS}
\label{sec:measurable-lts}
Many interesting (counter)examples can be constructed by considering
non-probabilistic NLMP, i.e., one $\nlmp S = (S,\calS, \{\Tfont{T}_a : a\in L \})$
such that for all $a\in L$ and $s\in S$, $\Tfont{T}_a(s)$ consists entirely
of point-masses (i.e., Dirac measures).  Assume we have such an
$\nlmp S$; we may write each set $\Tfont{T}_a(s)$ as $\{\delta_x : x\in
\Tfont{\tilde T}_a(s)\}$ where $\Tfont{\tilde T}_a(s) \subseteq S$ for each $s$ (moreover, it can be
seen that $\Tfont{T}_a(s)\in\calS$). Since there is a natural
correspondence between points $s\in S$ and Dirac measures $\delta_s
\in \dist(S)$, we may discard all references to $\dist(S)$ and work
with the simpler structure  $(S,\calS, \{\Tfont{\tilde T}_a : a\in L \})$ which is   essentially a
labelled transition system  (with some restrictions) over a measurable space.
This  presentation of non-probabilistic NLMP appears in
Wolovick \cite{Wolovick}. 
\begin{definition}\label{d:mlts}
A \emph{measurable labelled transition system} (MLTS) is a tuple
$\nlmp S = (S,\calS, \{\Tfont{\tilde T}_a : a\in L \})$ such that $(S,\calS)$ is a
measurable space and for each label $a\in L$, $\Tfont{\tilde T}_a:(S,\calS)\to
(\calS, H(\calS))$ is a measurable map.
\end{definition}
If we write $\pre{a}{Q}$
 for $\{s : \Tfont{\tilde T}_a(s) \cap Q \neq \emptyset\}$, 
then the measurability  requirement on $\Tfont{\tilde T}_a$ in
Definition~\ref{d:mlts} amounts asking 
$\calS$ to be \emph{stable} under the map $\pre{a}$: for all
$Q\in\calS$, $\pre{a}{Q}\in\calS$. 

Among the trivial examples of MLTS we might take any LTS $(S,
\{\Tfont{\tilde T}_a : a\in L \})$ (where $\Tfont{\tilde T}_a :S\to\Pow(S)$) and attach to it the powerset \sig-algebra;
then the measurability requirements are immediately satisfied. A more
interesting one is the following.
\begin{example}
There is a well-known  duality for $\mathrm{BAO}_\tau$, the variety of Boolean
algebras with operators of type  
$\tau$ \cite[p.\ 354ff]{Blackburn:2006:HML:1207696}. Consider the case
where the operators $\{\dia_a : a\in L\}$ in $\tau$ are all
unary. Then the dual category
consists of Kripke frames $(S,\{R_a : a \in L\})$, where $\dia_a =
\pre{a}$ for each $a$, endowed with a Stone topology $\mathcal{T}$ such that
$\dia_a$ maps clopen sets into clopen sets and the relations $R_a$ are
\emph{point-closed}, meaning that $R_a[s] \doteq \{y\in S : s\mathrel{R_a}
y \}$ is closed for each $s\in S$ (see
\cite[Sect.~4.6]{kracht1999tools}). For such a frame, we can define an associated LTS
$(S, \{\Tfont{\tilde T}_a : a\in L \})$, where $\Tfont{\tilde T}_a(s)
\doteq R_a[s]$, and if we augment it with the Borel \sig-algebra, the
previous restrictions ensure   that  $(S,\B(\mathcal{T}), \{\Tfont{\tilde T}_a
: a\in L \})$ is a MLTS.
\end{example}

We have 
the following  correspondence between non-probabilistic NLMP and
MLTS.
\begin{prop}[{\cite[Prop 4.7]{Wolovick}}]
  Assume that  $\calS$ is countably generated and separates points on
  $S$, 
  and for all $a\in L$ and $s\in S$, $\Tfont{T}_a(s)=\{\delta_x : x\in
  \Tfont{\tilde T}_a(s)\}$ for some sets $\Tfont{\tilde T}_a(s)\subseteq S$. Then
  $(S,\calS, \{\Tfont{T}_a : a\in L \})$   is a NLMP iff  $(S,\calS, \{\Tfont{\tilde T}_a : a\in L \})$ is a MLTS.
\end{prop}
The hypothesis on the measurable structure is satisfied, for example,
in any analytic Borel space $(S,\calS)$, and in the following sections we
will work only with such spaces.

The following Lemma from
\cite{D'Argenio:2012:BNL:2139682.2139685} provides translations for
our three notions of bisimulation. Event bisimulation is again a
restatement of the measurability criterion (now that of Definition~\ref{d:mlts})
for sub-\sig-algebras,
and the reader might notice that for MLTS the notion of traditional
bisimulation has a look closer to the standard definitions, e.g. as in
Kripke frames.
\begin{lemma}
  Let $\nlmp S$ be a NLMP such that for all $a\in L$ and $s\in S$, $\Tfont{T}_a(s)=\{\delta_x : x\in
  \Tfont{\tilde T}_a(s)\}$.
  \begin{enumerate}
  \item \label{c:non-prob-event}
    A $\sigma$-algebra $\Lambda\subseteq\Sigma$ is an event bisimulation
    on $\nlmp S$ if and only if it is stable under the mapping $\pre{a}$.
  \item \label{c:non-prob-state}
    A symmetric relation $R$ is a state bisimulation on $\nlmp S$ if and
    only if for all $s,t\in S$ such that $s\,R\,t$, it holds that for
    all $Q\in\calS(R)$, $s \in \pre{a}{Q} \sii t \in \pre{a}{Q}$.
  \item \label{l:non-prob-tradit}
    A symmetric relation $R$ is a traditional bisimulation on $\nlmp S$ if and only if for all
    $s,t\in S$ and $u\in \Tfont{\tilde T}_a(s)$, if $s\,R\,t$  then there exists
  $v\in \Tfont{\tilde T}_a(t)$ such that $u\rel{\calS(R)}v$.
  \end{enumerate}
\end{lemma}
Since for every relation one has $R\subseteq{\rel{\calS(R)}}$, 
standard (Milner's) bisimilarity $\sim_\mathrm{M}$ for LTS is also a traditional
bisimulation of MLTS, and then ${\sim_\mathrm{M}} \subseteq
{\tbisim}$. Traditional bisimilarity for MLTS is coarser than standard
bisimilarity because bisimilarity classes are not directly
``accessible'', but only through the measurable subsets. Also, examples in
\cite{DWTC09:qest,D'Argenio:2012:BNL:2139682.2139685} show that there
are (non-probabilistic)  image-uncountable NLMP where the notions
above are different. 

Observe that for a non-probabilistic NLMP as
above, the quantitative assessment $q$ in formulas  of the
type $\qmodalgq{\phi}{q}$ doesn't play any role, and hence $\calL$ is
  equivalent to  Hennessy-Milner
logic with countable conjunctions (and disjunctions) on this family of
processes%
\footnote{For instance, the Hennessy-Milner formula $\modal{a}\phi$ is
  equivalent to the $\calL$-formula $\modal{a}\qmodalgq{\phi}{1}$.}%
. Therefore, $\calL$ characterizes standard bisimilarity for image-countable
MLTS and then ${\rel{\sem\calL}} = {\sim_\mathrm{M}}$. By appealing to Theorem~\ref{theorem:logic_is_sound} we can state
\begin{prop}\label{pr:all-bisim-the-same}
For image-countable MLTS, all kinds of bisimilarities (traditional,
state, event, and standard) coincide.
\end{prop}
We shall henceforth drop the subindexes and use simply $\sim$.

Let's return to  the problem left open at the end of the previous
section. To apply Lemma~\ref{lm:anylogic_to_bisim}, the candidate logic should
satisfy several requirements. So the first question is if   there
actually exists any countable logic that characterizes bisimilarity
for  countable LTS. The answer is given by the following
\begin{example}[X. Caicedo]\label{ex:caicedo}
  Fix a countable set of labels $L$. There are at most $2^{\aleph_0}$
  (bisimilarity classes of) 
  countable LTS over $\N$. Hence there is an injective function $f$ from
  bisimilarity classes to $\Pow(\N)$. Our `logic' will consist of
  countably many  atomic formulas $P_n$ ($n\in\N$) with the
  following semantics:
  \[\nlmp S, s \models P_n \iff n\in f([(\nlmp S, s)]_{\sim}),\]
  where $[\cdot]_{\sim}$ denotes $\sim$-classes of equivalence. 
  The logic $\calL_X := \{P_n : n\in\N\}$ is sound and complete for
  bisimilarity.
\end{example}
The logic $\calL_X$ is devised in a non-constructive manner, and the
main result in this work is to show that actually the extensions of
formulas of such a countable logic cannot be Borel sets, confirming
the intuition that formulas in $\calL_X$ cannot be conceived as any
reasonable kind of ``test'' on a process, no matter how the function
$f$ is chosen.
\section{The main result}
We will use some concepts from sequence (zero-dimensional) spaces. The
reader may review these and related concepts in
Kechris~\cite{Kechris} or Moschovakis~\cite{moschovakisDesc}. Let
$E$ be a set. The
set of all finite sequences of elements of $E$ will be denoted by
$E^*$. The empty sequence will be denoted by $\vacia$. The $i$th
element of a sequence $s\in E^*$ will be denoted by
$s^i$; hence $s=\<s^0,\dots, s^{|s|-1}\>$, where $|s|$ is the
length of $s$. The concatenation of two sequences $s,t\in E^*$
will be denoted by $s\ap t$; in case $t=\<e\>$ with $e\in E$, we will
write  $s\ap e$ instead of $s\ap \<e\>$. A \emph{tree} on $E$ is a
subset of $E^*$ closed by taking prefixes. We will be  interested
in the case where $E$ is countable, and specially $E=\N$.

Let $A$ be countable, and consider the discrete topology on it. The
product space $A^\N$ of all infinite sequences of 
elements of $A$ is Polish and has a (clopen) basis given by the sets $C_f =
\{x\in A^\N: f\subset x\}$, where $f$ is a finite function. When $A=2$,
we obtain the Cantor space $\Pow(\N)$. In general, for every countable
set $B$, we regard $\Pow(B)$ as a separable compact Hausdorff space where its
topology has the following subbasic sets:
\begin{equation}\label{eq:4}
\{X\subseteq B : s\in X \},\ \{X\subseteq B : t\notin X \}\qquad
\text{with } s,t\in B.
\end{equation}
Hence the basic
open sets are $\{X\subseteq B : P\subseteq X \et N\cap X =\emptyset\}$, where
$P,N\subseteq B$ are finite; it is also immediate that the Borel
\sig-algebra of $\Pow(B)$ is generated by sets of the first kind in~(\ref{eq:4}). In case $B=\N$ we obtain exactly the
basis given by $\{C_f\}$. We will use the  homeomorphic spaces $2^B$
and $\Pow(B)$ (and their respective presentations) interchangeably.

A binary structure $(\N,R)$ can be represented by a point in
$\mathrm{Rel} \doteq 2^{\N\times\N}$. The set $LO$ of strict linear
order relations,
\[LO \doteq \{ R \in  2^{\N\times\N} : R \text{ is irreflexive,
  transitive, and total: } \forall x\neq y,\ x\mathrel{R}y \text{ or }
y\mathrel{R}x\},\] 
is a closed subset of $\mathrm{Rel}$ \cite[Sect.\ 16.C and 27.C]{Kechris}; this
can be easily seen since universal conditions (as, for
instance, ``$\forall x : x\mathrel{\centernot{R}}x$'') define closed sets. Hence
$LO$ is a Polish space. The same happens to the  set $\Tree$ of all
trees on $\N$: it is a closed subset of  $2^{\Suc}$.

Let $X$ be a 
Polish space; recall we are using  $\B(X)$ for its Borel sets.
$\Sig_1^1(X)$ denotes the  family 
\[\{A\subseteq X : \exists Y \text{ Polish and }f:Y\to X
\text{ continuous with }f[Y]=A\}\]
of analytic subsets of $X$ (see \cite{Kechris}). Both Borel and
analytic sets are preserved by taking continuous preimages and countable unions
and intersections. We will usually omit the reference to the space $X$.
\subsection{Bisimilarity on denumerable trees}
We first recall a standard construction of trees from linear orders. 
Given a strict linear order $\E = (E,R)$ over a countable set $E$ 
we may define a new countable structure $(T_\E,{\pred})$ 
as follows:
\begin{itemize}
\item $T_\E\doteq \{s\in E^* :
  s^{|s|-1}\mathrel{R}s^{|s|-2}\mathrel{R}\dots\mathrel{R}s^0\}\cup\{\<e\>
  : e\in E\}\cup\{\vacia\}.$
\item $s\mathrel{\pred}s' \iff  \exists  e\in E:s\ap e = s'$. 
\end{itemize}
The tree $T_\E$ consists of all finite decreasing
sequences in $(E,R)$. We also use $T_R$ to denote this tree,
whenever $E$ is clear from the context. We can obviously regard the
binary structures $(T_\E,\pred)$ as a LTS with a singleton label set
$L\doteq \{l\}$, and we say that two trees are bisimilar if there is
a bisimulation relating both of the roots. Our next landmark will be to prove that the relation
of bisimilarity among the processes  $T_\E$ is a non Borel relation on
$\Tree$. 

We'll now sketch the argument of the proof. Let $WO$ ($\NWO$) be the
set of (non) wellorder relations on $\N$, regarded as subsets
of $LO$. It is well known that the set $\NWO$ is a 
$\Sig^1_1$-\emph{complete} set
, in the sense that it is as `complicated' as any
analytic subset of a (zero-dimensional) Polish space; in particular,
it is not Borel. We will be able to distinguish elements $R\in \NWO$
among linear orderings essentially just by looking at the bisimilarity
type of the tree $T_R$ over $\N$, regarded as a processes with initial state
$\vacia$. Thus we'll have succeeded \emph{reducing} $\NWO$ to
the relation of bisimilarity, thus showing that the latter is not
Borel.
\medskip

In first place we elucidate the notion of reduction that concerns us.
\begin{definition}
Let $X,Y$ be Polish spaces and $A\subseteq X$ and
$B\subseteq Y$. A \emph{continuous reduction} of $A$ to $B$ is a
continuous map $f:X\to Y$ such that $f^{-1}[B]=A$; in this case we say
that $A$ is \emph{Wadge reducible} to $B$.
$B$ is \emph{$\Sig_1^1$-hard}
if for every zero-dimensional Polish space $X$ and every
$A\in\Sig_1^1(X)$, $A$ is Wadge reducible to $B$, and $B$ is
\emph{$\Sig_1^1$-complete} if moreover $B\in\Sig_1^1(X)$.
\end{definition}
\begin{theorem}[Lusin, Sierpi\'nski {\cite[27.12]{Kechris}}]
$\NWO$ is a $\Sig_1^1$-complete and, in particular, non Borel subset
  of $LO$.
\end{theorem}
There is an extensive literature on  $\Sig_1^1$-complete  sets; one of
the prominent examples is the relation of isomorphism on denumerable
structures, and in particular for binary structures.
\begin{theorem}[{\cite[27.D]{Kechris},\cite{Kechris-errata}}\label{th:iso_sigma11_complete}]
    The relation of isomorphism between denumerable binary structures, coded as a
    subset of $2^{\N\times\N}\times2^{\N\times\N}$ is
    $\Sig_1^1$-complete.%
    \footnote{Another examples of (Borel) $\Sig_1^1$-complete
      \emph{relations} are those of isomorphism and embedding between
      separable Banach spaces, and the relation ``$F$ and $F'$ have
      nonempty intersection'' in the space of closed subsets of $\N^\N$.}%
\end{theorem}

Looking into the detail of the argument, we argue the trees $T_R$
corresponding to $R\in WO$ are 
well founded: regarded as  processes, they're terminating. The depth
of such a 
tree equals the length of the wellorder $R$ and since this depth can be `measured'
by using modal formulas, \emph{nonisomorphic wellorders have non-bisimilar
trees.} 

In the case of an order  relation $S$ on $\N$ that is 
not a wellorder, there must be an infinite branch in $T_S$ (and hence
it can't be bisimilar to any $T_R$ with $R\in WO$), but we can say more: the
tree $T_S$ is bisimilar to the process that results after  attaching a loop to the
initial state of the tree $T'$ associated to the maximal
well-ordered initial segment of $(\N, S)$. For the purpose of having a
manageable example, consider the linear orders $\mathbf{L}\doteq(\{0,1,2\},<_3)$ and
$\mathbf{L'}\doteq(\{0,1,2\}\cup\{\dots,-3,-2,-1\},<')$ where $a <' b$ if $b$ is
negative and $a$ is not, and otherwise $<'$ behaves as the ordering
of the integers. Then the trees corresponding to each of them are respectively
bisimilar to the processes in the following picture:
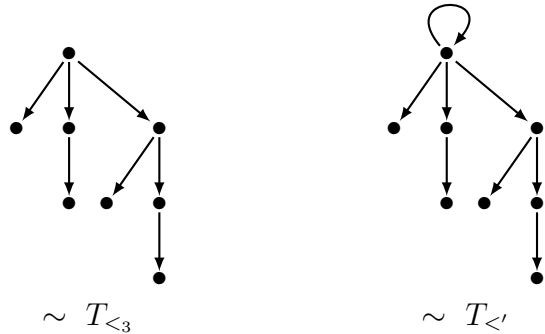
\begin{center}
\begin{tabular}{ccc}
\begin{tikzpicture}
  [>=latex, thick,
    nodo/.style={thick,minimum size=0cm,inner sep=0cm}]
  \node (root) at (0,0) [nodo] {$\bullet$} ;
  \node (0) at (-0.7,-1) [nodo]  {$\bullet$} ;
  \node (1) at (0,-1)  [nodo] {$\bullet$} ;
  \node (10) at (0,-2)  [nodo] {$\bullet$} ;
  \node (2) at (1.2,-1)  [nodo] {$\bullet$} ;
  \node (20) at (0.5,-2)  [nodo] {$\bullet$} ;
  \node (21) at (1.2,-2)  [nodo] {$\bullet$} ;
  \node (210) at (1.2,-3)  [nodo] {$\bullet$} ;
  
  \draw [->] (root) edge  (0);
  \draw [->] (root) edge  (1);
  \draw [->] (root) edge  (2);
  \draw [->] (1) edge  (10);
  \draw [->] (21) edge  (210);
  \draw [->] (2) edge  (21);
  \draw [->] (2) edge  (20);
\end{tikzpicture}
&
\hspace{5em}
&
\begin{tikzpicture}
  [>=latex, thick,
    nodo/.style={thick,minimum size=0cm,inner sep=0cm}]
  \node (root) at (0,0) [nodo] {$\bullet$} ;
  \node (0) at (-0.7,-1) [nodo]  {$\bullet$} ;
  \node (1) at (0,-1)  [nodo] {$\bullet$} ;
  \node (10) at (0,-2)  [nodo] {$\bullet$} ;
  \node (2) at (1.2,-1)  [nodo] {$\bullet$} ;
  \node (20) at (0.5,-2)  [nodo] {$\bullet$} ;
  \node (21) at (1.2,-2)  [nodo] {$\bullet$} ;
  \node (210) at (1.2,-3)  [nodo] {$\bullet$} ;
  
  \draw [->] (root) edge[out=130, in=50, looseness=13]  (root);
  \draw [->] (root) edge  (0);
  \draw [->] (root) edge  (1);
  \draw [->] (root) edge  (2);
  \draw [->] (1) edge  (10);
  \draw [->] (21) edge  (210);
  \draw [->] (2) edge  (21);
  \draw [->] (2) edge  (20);
\end{tikzpicture}\\
$\sim \ T_{<_3}$ & & $\sim \ T_{<'}$
\end{tabular}
\end{center}
%
We will show that  the bisimilarity type of $T_S$ for non-well founded $S$
only depends on the order type of the maximal well-ordered initial
segment of $(\N, S)$.
Continuing with the previous example, observe that 
$\mathbf{L'}$ is the ordered sum of the poset 
$\mathbf{L}$ and the poset $\mathbf{M}$ of the negative integers; we denote
this ordered sum as $\mathbf{L}+\mathbf{M}$. Although the posets
$\mathbf{L'}$ and $\mathbf{L'}+\mathbf{L'}$  are not isomorphic, they
have the same   maximal well-ordered initial segment (i.e., $\mathbf{L}$)
and hence $T_\mathbf{L'}$ and $T_{\mathbf{L'}+\mathbf{L'}}$
are bisimilar.


Now we may apply the following idea by Dougherty used in one proof
(outlined in
\cite{Kechris-errata}) of Theorem~\ref{th:iso_sigma11_complete}. If
$\mathbf{L}$ is a wellorder, then 
$\mathbf{L}+\mathbf{L}$ is itself a 
wellorder and $T_{\mathbf{L}} \nsim T_{\mathbf{L}+\mathbf{L}}$ since $\mathbf{L}\not\iso\mathbf{L}+\mathbf{L}$. But in case $\mathbf{L}$ is not, the trees $T_{\mathbf{L}}$ and $T_{\mathbf{L}+\mathbf{L}}$ are
indeed bisimilar since $\mathbf{L}$ and $\mathbf{L}+\mathbf{L}$ have the same maximal well-ordered
initial segment.

Then the reduction we are looking for is given by the map $R\mapsto
(T_R,T_{R+R})$, where $R\in LO$.
\bigskip

We'll begin by proving that this map is indeed continuous.
\begin{lemma}\label{l:TR-cont}
  The map $R\mapsto T_R$ from $LO$ to
  $2^\Suc$ is continuous.
\end{lemma}
%
\begin{proof}
  We will show that the function from $2^{\N\times\N}$ to $2^\Suc$
  with the same definition is continuous. It is enough to prove that
  the preimage by $T_\sbt$ of subbasic elements are open.
  Assume $s\in \Suc$. Then 
  \[ R \in (T_\sbt)^{-1}[\{A\subseteq 2^\Suc : s \in A\}] \iff s\in T_R
  \iff  \forall j : 0<j<|s| \ent s^j\mathrel{R}s^{j-1}.\]
  Hence the preimage can be written in the form 
  \[(T_\sbt)^{-1}[\{A\subseteq 2^\Suc : s \in A\}]  = \bigl\{R\subseteq {\N\times\N} :
  \{(s^{|s|-1},s^{|s|-2}),\dots,(s^1,s^0)\} \subseteq R\bigr\},\] 
  which is basic open. The other case is very similar. For $t\in \Suc$,
  \[ 
  t \notin T_R
  \iff  \exists j : 0<j<|t| \et t^j\mathrel{\centernot{R}}t^{j-1},\]
  hence we have
  \[(T_\sbt)^{-1}[\{A\subseteq 2^\Suc : t \notin A\}] =
  \bigcup_{0<j<|t|}\{ R \subseteq  \N\times\N :  (t^j,t^{j-1})\notin
  R\},\] 
  again an open set.
\end{proof}
We'll need a definition of sum of linear orders such that $LO$ is
closed under this operation. Given $R,R'\in LO$, let
\[(R+R')(n,m)\doteq \begin{cases}
                     1                            & 2\mid n \et 2\nmid m \\
                     R(\tfrac{n}{2},\tfrac{m}{2}) & 2\mid n \et 2\mid m\\
                     R'(\tfrac{n-1}{2},\tfrac{m-1}{2}) & 2\nmid n \et 2\nmid m \\
                     0 & 2\nmid n \et 2\mid  m. 
                     \end{cases}\]
Hence we have the following straightforward lemma, and its corollary
rounds out the proof of continuity.
\begin{lemma}\label{l:sum-continuous}
$(\N,R+R')$ is isomorphic to the ordered sum of $(\N,R)$ and
  $(\N,R')$, and $(R,R')\mapsto R+R'$ is continuous from $LO\times LO$
  to $LO$.
\end{lemma}
\begin{corollary}\label{c:reduction-cont}
The map  $R\mapsto
(T_R,T_{R+R})$ is continuous from $LO$   to $\Tree\times\Tree$.
\end{corollary}
Now we want to check that wellorders of different type can't give rise to
bisimilar trees. For that we define by recursion on $\alpha<\om_1$ the following 
modal formulas:
\begin{itemize}
\item $\phi_0\doteq \top$.
\item $\phi_{\alpha+1}\doteq \dia \phi_\alpha$.
\item $\phi_\lda\doteq \bigwedge_{\beta<\lda}\phi_\beta$, for limit $\lda$.
\end{itemize}
When satisfied at the root of a tree, formula $\phi_\alpha$ states
that its depth is as least $\alpha$.  We record two standard results.
\begin{prop}
  For a wellorder $(\N,R)$ of type $\alpha$,   $T_R, \vacia \models
  \phi_\beta$ if and only if  $\beta\leq\alpha$. If  $(\N,R)$  is
  not well founded,  $T_R, \vacia \models \phi_\beta$ for all $\beta<\om_1$.
\end{prop}
\begin{corollary}\label{c:bisim-separ-WO}
 If $(\N,R)$ is a wellorder and $R'\in LO$, $T_R, \vacia  \sim T_{R'}, \vacia \iff
 (\N,R)\iso(\N,R')$. In particular,  $T_R, \vacia  \nsim T_{R+R}, \vacia$.
\end{corollary}
%
Now we turn to the proof that if $\mathbf{L'}$ is not a wellorder,
then the trees $T_\mathbf{L'}$ and $T_{\mathbf{L'}+\mathbf{L'}}$ are
bisimilar. For this we  will characterize  
the process $T_{\mathbf{A}+\mathbf{B}}$ corresponding to the ordered sum of
posets $\mathbf{A}$ and $\mathbf{B}$ by using a kind of \emph{disrupt}
construction between $T_{\mathbf{A}}$ and $T_{\mathbf{B}}$. 
More generally, take two processes (represented by binary
structures) $\mathbf{C} = (C,R)$ and $\mathbf{D}=(D,Q)$  and 
``execute them in succession'': taking $c\in C$ and $d\in D$, we define
a new state $c\disr d$ that starts behaving as $c$ but in any moment can be
``disrupted'' and then it behaves as $d$%
   \footnote{For the process-algebra minded, the operator
     $\disr$ is very much like LOTOS disabling
     operator~\cite{DBLP:journals/cn/BolognesiB87}.}%
. Formally, let  two binary
structures $\mathbf{C}$ and $\mathbf{D}$ as above be given with $C$ and $D$
disjoint, and define $\mathbf{C}\disr\mathbf{D}$ to be a new binary 
structure with universe $(C\times D) \cup D$ (where we write an
ordered pair $(c,d)$ as $c\disr d$) equipped with the least binary relation $\to$
such that
\begin{align}
  c\disr d\,\to\, c'\disr d &\iff c\mathrel{R}c'\label{eq:1}\\
  c\disr d\,\to\, d' &\iff d\mathrel{Q}d'\label{eq:2}\\
  d\,\to\, d' &\iff d\mathrel{Q}d'\label{eq:3},
\end{align}
for all $c,c'\in C$ and $d,d'\in D$. In the following Lemma, the
binary structures $\mathbf{C}$ and $\mathbf{D}$  considered are trees
with $R$ and $Q$ both being the successor relation $\prec$.
\begin{lemma}\label{l:sum-to-disr}
Let  $\mathbf{A} = (A,R)$ and $\mathbf{B}=(B,T)$  be strict linear
orders with $A$ and $B$ disjoint. Then ${T_{\mathbf{A}+\mathbf{B}},\vacia}\; \sim\;
{T_{\mathbf{B}}\disr T_{\mathbf{A}},\vacia\disr\vacia}$.
\end{lemma}
\begin{proof}
  Observe that  any decreasing sequence in $\mathbf{A}+\mathbf{B}$
  decomposes uniquely as the concatenation of a sequence $s$ in $\mathbf{A}$
  after a sequence $t$ in $\mathbf{B}$. We'll check that  the relation
  \[\th \doteq \{(t,t\disr\vacia) : t\in T_{\mathbf B}\} \cup \{(t\ap s,
  s): t\in T_{\mathbf B}, s\in T_{\mathbf A}\setminus\{\vacia\}\}\]
  is a bisimulation. Note that $(\vacia,\vacia\disr\vacia)\in\th$.

  We  check the forth and back property for each of type of pairs in
  $\th$. Assume $(t,t\disr\vacia)\in\th$, hence $t\in
  T_{\mathbf{B}}$ (i.e., it is a decreasing sequence in
  $\mathbf{A}+\mathbf{B}$ with no elements from $A$). Further, assume
  we have a ``transition'' 
  $t\prec t\ap b$ in $T_{\bA+\bB}$ with  $b\in B$ (this means that $t\ap
  b\in T_\bB$). The forth property 
  follows immediately since $t\disr \vacia \to t\ap b\disr \vacia$ by
  (\ref{eq:1}) and $( t\ap b,t\ap b\disr \vacia)\in \th$. In case of a
  transition $t\prec t\ap a$ with $a\in A$, we first see that
  $t\disr\vacia\to \<a\>$ by (\ref{eq:2}) (with $c=t$, $d=\vacia$, and
  $d'=\<a\>$). But now $(t\ap a, \<a\>)\in \th$, and the forth property
  follows. The back property is checked by using very similar
  arguments.

  Now assume $(t\ap s, s)\in \th$ where $s\in T_\mathbf{A}$
  and $t\in T_\mathbf{B}$. Then 
  \[t\ap s \prec t\ap s\ap a \iff s \prec s\ap a \iff s\to s\ap a\]
  for each $a\in A$, where the last equivalence holds  by (\ref{eq:3}) (with $d=s$,
  $d'=s\ap a$). This gives us the forth and back property for
  this case.
\end{proof}
\begin{lemma}\label{l:no-first-element}
If $\mathbf{A} = (A,R)$ is a countable strict linear
order without  first element, then $(T_{\mathbf{A}},\pred,\vacia)\iso
(\Suc,\pred,\vacia)$. In particular, $T_{\mathbf{A}},\vacia \sim \Suc,\vacia$.
\end{lemma}
\begin{proof}
It is enough to note that under the hypothesis, $T_{\mathbf{A}}$ is a (countable)
tree in which every node has infinitely many successors. 
\end{proof}
\begin{corollary}\label{c:NWO-R+R}
If $R\in\NWO$, $T_R,\vacia \sim T_{R+R},\vacia$.
\end{corollary}
\begin{proof}
Decompose $\mathbf{N}\doteq(\N,R)$ as $\mathbf{A} + \mathbf{B}$,  where $\mathbf{A}$ is the 
maximal well-ordered initial segment of $\mathbf{N}$ and hence
$\mathbf{B}$ has no first element. Then
$\mathbf{N}+\mathbf{N}\iso\mathbf{A} + \mathbf{B}+\mathbf{A} +
\mathbf{B}\iso \mathbf{A} + \mathbf{C}$, where $\mathbf{C} \doteq \mathbf{B}+\mathbf{A} +
\mathbf{B}$ has no
first element. Hence
\begin{align*}
   T_{R+R},\vacia & \sim T_{\mathbf{A}+\mathbf{C}},\vacia 
     && \text{since }T_{R+R} = T_{\mathbf{N}+\mathbf{N}} \iso
   T_{\mathbf{A}+\mathbf{C}},\\
   & \sim   {T_{\mathbf{C}}\disr T_{\mathbf{A}},\vacia\disr\vacia} 
     && \text{by Lemma~\ref{l:sum-to-disr},}\\
   & \sim
   {T_{\mathbf{B}}\disr T_{\mathbf{A}},\vacia\disr\vacia} 
     && \text{by Lemma~\ref{l:no-first-element}, }T_{\mathbf{C}}\iso T_{\mathbf{B}}\\
   & \sim
   T_{\mathbf{A}+\mathbf{B}},\vacia 
     && \text{by  Lemma~\ref{l:sum-to-disr},}\\
   & \sim T_{R},\vacia.
\end{align*}
\end{proof}
At this point we can gather all the previous lemmas and prove the
following technical result, which is the main ingredient for the
developments in the next section.
\begin{theorem}
Bisimilarity on $\Tree$ is $\Sig_1^1$-hard, and hence not Borel.
\end{theorem}
\begin{proof}
We have seen that $R\mapsto f(R) \doteq (T_R,T_{R+R})$  is continuous
by Corollary~\ref{c:reduction-cont}. 
By Corollary~\ref{c:bisim-separ-WO} and Corollary~\ref{c:NWO-R+R},
$f^{-1}[{\sim}] = \NWO$ and hence $\NWO$ is Wadge reducible to $\sim$.
This makes $\sim$ $\Sig_1^1$-hard, and not Borel (since Borel sets
are preserved by continuous preimages).
\end{proof}
\subsection{Bisimilarity on MLTS is not Borel}
In this section we will obtain our main result. For this, we will
construct a MLTS $\nlmp F$ for which the relation of bisimilarity is
not Borel. It is easy to show  that for any countable measurable
logic and for any MLTS $\nlmp S$, the induced relation of logical equivalence
on $\nlmp S$ is Borel, so the punchline is that there is no such a
logic for any class of MLTS containing $\nlmp F$.

The state space of $\nlmp F$ will be the result of collecting  all  the tree
processes $T_R$ (with an appropriate ``tagging''), in such a way that
bisimilarity on $\nlmp F$ is exactly the relation of bisimilarity on
$\Tree$. The only details to be taken care of are essentially the
measurability requirements to be a MLTS.

Let $L\doteq\{l\}$ be the singleton label set. Define $\nlmp
F=(F,\B(F), {\bpred})$ (where we are taking $\Tfont{\tilde T}_l = {\bpred}$) such that
\begin{itemize}
\item $(F,\B(F)) \doteq (\Tree\times\Suc,\B(\Tree\times\Suc))$
  (where $\Suc$ is considered discrete). Note that this is a Polish space.
\item $\mathop{\bpred}(T,s) \doteq \{(T,s')$ : $s,s'\in T \et s\pred s'\}$.
\end{itemize}
We prove that $\nlmp F$ is a MLTS. First note that sets
$\mathop{\bpred}(T,s)$, being countable, are  Borel. We will also use 
the symbol $\bpred$ also as a binary relation, defined in the obvious
way: $(T,s) \bpred (T',s')$ iff $T=T'$, $s,s'\in T$ and $s\pred s'$.


  
 For a subset $A\subseteq X\times Y$ of a product and $c\in Y$ the
 \emph{section} $A|_c$ is the set $\{x\in X : (x,c)\in A\}$, the
 preimage of the injection $x\mapsto (x,c)$. 
\begin{lemma}
   $\pre{l}{Q}$ is Borel for each $Q\in\B(\Tree\times\Suc)$, and hence
  $\nlmp F$ is a MLTS.
\end{lemma}
\begin{proof}
  \begin{align*}
    \pre{l}{Q} & = \{(T,s)\in F : \exists (T',s')\in Q\ ((T,s)\bpred
    (T',s') )\} \\
    & =  \{(T,s)\in F : \exists s'\ ((T,s')\in Q \et (T,s)\bpred
    (T,s') )\} \\
    & =  \{(T,s)\in F : \exists n\in\N(s\ap n \in T \et (T,s\ap n)
    \in Q) \} \\
    & =  \bigcup_{n\in\N}\bigcup_{s\in\Suc} \{(T,s)\in F :s\ap n
    \in T \et (T,s\ap n )\in Q \}
  \end{align*}
  Now we may write the set inside the unions (now for fixed $s,n$) as
  \begin{align*}
    \{(T,s)\in F :s\ap n     \in T \et (T,s\ap n )\in Q \}  &= 
  \Bigl( Q  \cap \{(T,s\ap n) :s\ap n \in T\}\Bigr)|_{s\ap n} \times
  \{s\} \\
  & =  \Bigl(Q  \cap \bigl(\{T :s\ap n \in T\} \times \{s\ap n\}\bigl)
  \Bigr)|_{s\ap n} \times  \{s\}
  \end{align*}
  The inner rectangle is clopen, and since $Q$ is Borel, the set between
  the big parentheses is Borel. The whole set is easily Borel, too.
\end{proof}
%
%
\begin{theorem}\label{t:bisim-not-borel}
  The relation of bisimilarity  is a $\Sig_1^1$-hard subset of $F\times
  F$, and hence not Borel. 
\end{theorem}
\begin{proof}
We will reduce again $\mathit{NWO}$ to bisimilarity. By
Proposition~\ref{pr:all-bisim-the-same} we can consider standard 
bisimilarity, since $\nlmp F$ is image-countable. It is immediate that
states $(T,\vacia)$ and $(T',\vacia)$ are bisimilar in $\nlmp F$ if
and only if there is a bisimulation between the tree processes $T$ and $T'$.
Since the injection $T\mapsto (T,\vacia)$ is continuous
from $\Tree$ to $\Tree\times\Suc$, the composition $R\mapsto
f(R) \doteq ((T_R,\vacia),(T_{R+R},\vacia))$ also is (by
Lemmas~\ref{l:TR-cont} and~\ref{l:sum-continuous}). This $f$ is a
suitable reduction, since $f^{-1}[{\sim}] = \NWO$.
\end{proof}
%

We arrive at the main result of this work.
\begin{theorem}
There is no countable logic $\calL$ that characterizes bisimulation on
$\nlmp F$ such that  $\sem{\calL}\subseteq \B(F)$.
\end{theorem}
\begin{proof}
  Assume ${\rel{\sem\calL}}={\sim}$. Then 
  \begin{equation*}\label{eq:bisim-logic}
    s\sim t \iff  (s,t)\in\bigcap \Bigl\{(\sem{\phi}\times  \sem{\phi}) \cup
    \bigl((F\setminus\sem{\phi})\times (F\setminus \sem{\phi})\bigr) : \phi\in \calL\Bigr\}
  \end{equation*}
This contradicts Theorem~\ref{t:bisim-not-borel}, since the right-hand
side is a Borel definition of $\sim$.
\end{proof}
\subsection{Bisimilarity is $\Sig_1^1$-complete}

We finally show in this section that bisimilarity on $\nlmp F$ behaves similarly to
the isomorphism relation on countable structures: it is an analytic
equivalence relation with 
Borel classes. Since we already proved it to be $\Sig_1^1$-hard, we
would have seen it is a complete analytic set.

We will need a technical tool that allows us to obtain a
\emph{canonical representative} of the bisimilarity type of a tree
(see \cite[p.~275]{Blackburn:2006:HML:1207696}, Corollary 47 and the
paragraphs before). That is, for each tree $T$ we obtain a new
tree $\nlmp\Omega_{T}$ such that $T,\vacia \sim 
T',\vacia$ if and only if  $\nlmp\Omega_{T}\iso
\nlmp\Omega_{T'}$. We will essentially show that the map
$T\mapsto\nlmp\Omega_{T}$ is continuous, thereby reducing the 
relation of bisimilarity on $\Tree$ (and on $\nlmp F$) to isomorphism
of countable structures.

  An \emph{$\om$-indexed path from $s\in S$} on a LTS $\nlmp S = (S,R)$ is a sequence $u$ of the form
  \[u = s_0(s_1,a_1)(s_2,a_2)\dots(s_n,a_n)\]
  such that $s_0= s$, $a_i\in\N$ for all $i$, and $(s_{i-1},s_i)\in R$ for
  $i=1,\dots,n$.   
  The \emph{$\om$-expansion at $s$} of a LTS $\nlmp S$ is the LTS
  $\nlmp{\bar\Omega}_{\nlmp S}(s) = (\bar\Omega,\bar R)$ such that
  $\bar\Omega$ is the set of all $\om$-indexed paths on $\nlmp S$ from
  $s$  and the 
  relation $\bar R$ is defined by $(u,v)\in \bar R$ iff $v$ has the form
  $u(s,a)$ for some $a$ and $s$.

Since we are dealing with trees on $\N$, the latter construction provides  us
with  another tree that it is easily seen to be isomorphic to  the one given by the
following alternative description.
%
\begin{definition}
  The \emph{$\om$-expansion of  $ (T,\pred)$ at $s$}  is the LTS
  $\nlmp\Omega_{ T}(s)= (\Omega_{ T}(s),\bar R_{ T}(s))$ such that
  $\Omega_{ T}(s)= \{(t,n) : s\subseteq t \in T \et n \in \om\}$ and the
  relation $\bar R_{ T}(s)\subseteq (T\times\N)^2$ is given by
  \[(u,n) \mathrel{\bar R_{ T}(s)} (t,m) \iff u\pred t.\]
\end{definition}
The importance of this construction lies in the fact that two states in
a countable tree are bisimilar if and only 
if they have isomorphic $\om$-expansions. Note that the relation $\bar
R_{ T}(s)$ can be defined uniformly for all ${T}\in \Tree$ and
all $s\in T$. For this we take care of the requirement ``$\bar
R_{ T}(s)\subseteq (T\times\N)^2$'' by writing
\[(u,n) \mathrel{\bar R_{ T}(s)} (t,m) \iff s\subseteq u \et
u\pred t \et t \in T,\] 
since the right hand side implies $s,u\in T$.
Hence we are considering the function $\bar R:\Tree\times\Suc\to
2^{(\Suc\times\N)^2}$ defined by $\bar R(T,s) \doteq \bar
R_{ T}(s)$. Note that  $\bar R(T,s) = \emptyset$ if $s\notin T$.

  
\begin{lemma}\label{l:classes_onF_Borel}
  Bisimilarity classes on $\nlmp F$ are Borel.
\end{lemma}
\begin{proof}
  By the previous observations we conclude that for each $(T,s)\in F$,
  the bisimilarity class
  $[(T,s)]_{\sim}$ is mapped by $\bar R$  into the isomorphism class
  $[\bar R(T,s)]_{\iso}$. Then $[(T,s)]_{\sim} =
  \bar R^{-1}([\bar R(T,s)]_{\iso})$. By Scott's Theorem \cite{Scott64}, we
  know that isomorphism classes of countable (binary) structures are
  Borel, hence $[\bar R(T,s)]_{\iso}$ 
  is a Borel subset of $2^{(\Suc\times\N)^2}$. Then we just have to show
  that the map 
  %
  %
  $\bar R$
  is Borel measurable. We'll actually see that it is continuous.
  

  %
  It is enough to show that preimages of subbasic sets are open. Take
  $((u,n),(t,m))\in (\Suc\times\N)^2$; we have two cases. If  $u\pred t$,
  \begin{align*}
    \bar R^{-1}[\{R: ((u,n),(t,m))\in R\}]&=\{(T,s): t\in T \et s\subseteq
    u\} \\
    &= \{T:t\in T\}\times\{s: s\subseteq  u\},
  \end{align*}
  and it is empty otherwise. The last set is an open
  rectangle. 

  Now we go for the other type of subbasic open sets. If $u\npred t$,
  \[\bar R^{-1}[\{R: ((u,n),(t,m))\notin R\}] = \Tree\times\Suc.\]
  Otherwise, 
  \begin{align*}
    \bar R^{-1}[\{R: ((u,n),(t,m))\notin R\}]&=\{(T,s): t\notin T
    \text{ or } s\nsubseteq
    u\} \\
    &= \bigl(\{T:t\notin T\}\times\Suc\bigr)\cup \bigl(\Tree\times\{s: s\nsubseteq  u\}\bigr),
  \end{align*}
  again an open set.
\end{proof}

By the proof of the previous lemma, $\bar R$ is a reduction showing
that  bisimilarity
is $\Sig_1^1$, since isomorphism is. We also give a direct proof of
this fact, by analyzing an explicit definition of $\sim$ on $\nlmp
F$. We need an auxiliary calculation first.
\begin{lemma}\label{l:indexed_borel}
Let $A$ be countable with the discrete topology, $Y$ Polish and
$B_k\subseteq Y$  Borel for all $k\in  A$. Then
\[C(R,y) \stackrel{\sbt}{\iff} \forall  k\in R: (y\in B_k)\]
is Borel in $2^A\times Y$.
\end{lemma}
\begin{proof}
We have $(R,y)\in C \iff \forall k\in A : (k\in R \ent y\in B_k)$. Then 
\begin{align*}
C =& \bigcap_{k\in A}\{(R,y) : k\in R \ent y\in B_k\} \\
 =& \bigcap_{k\in A}\{(R,y) : k\notin R\} \cup \{(R,y) :  y\in
 B_k\} \\
 =& \bigcap_{k\in A} (\{R : k\notin R\}\times Y) \cup (2^A \times 
 B_k)
\end{align*}
which is obviously Borel.
\end{proof}
%
\begin{theorem}
Bisimilarity on $\nlmp F$ is $\Sig_1^1$.
\end{theorem}
\begin{proof}
As usual, $n,m$ denote  non negative integers and $s_i$ finite
sequences. The definition of bisimilarity on $\nlmp F$ is as follows:
\begin{multline*}
  (T_1,s) \sim   (T_2,s') \iff   \exists R\in
  2^{\Suc\times\Suc} : (s,s')\in R \et\\
 \et   \forall s_1 \forall s_2 \forall n . \Bigl(s_1\ap n \in T_1 \et s_2 \in T_2
  \et (s_1,s_2)\in R \implies \\
  \exists m : s_2\ap m \in T_2 \et (s_1\ap n, s_2\ap m) \in R \Bigr)
  \et \\
   \et \forall s_1 \forall s_2 \forall n . \Bigl(s_1\in T_1 \et s_2\ap n \in T_2
  \et (s_1,s_2)\in R \implies \\
  \exists m : s_1\ap m \in T_1 \et (s_1\ap m, s_2\ap n) \in R \Bigr).
\end{multline*}
It suffices to prove that the set defined inside the outer existential
quantifier is Borel in $2^{\Suc\times\Suc}\times F\times F$. We first
consider the third line of the definition. The set defined by
\[(R,(T_1,s),(T_2,s')) \in X_{(s_1,s_2),n,m} \mathrel{\stackrel{\sbt}{\iff}}  s_2\ap m \in T_2 \et (s_1\ap n, s_2\ap
m) \in R\]
is easily Borel. Then  the condition $\exists m : s_2\ap m \in T_2 \et
(s_1\ap n, s_2\ap m) \in R$ also is and
\[ \forall n . \Bigl(s_1\ap n \in T_1 \et s_2 \in T_2
  \implies
  \exists m : s_2\ap m \in T_2 \et (s_1\ap n, s_2\ap m) \in R \Bigr)\]
finally defines a Borel set of tuples $(R,(T_1,s),(T_2,s'))$
indexed by elements $(s_1,s_2)\in R$. We
may apply now Lemma~\ref{l:indexed_borel} and conclude that
\[\forall (s_1,s_2)\in R: \forall n . \Bigl(s_1\ap n \in T_1 \et s_2 \in T_2
  \implies
  \exists m : s_2\ap m \in T_2 \et (s_1\ap n, s_2\ap m) \in R \Bigr),\]
which is equivalent to the second and third lines of our definition
for bisimilarity, is Borel.

The rest of the formula is handled similarly.
\end{proof}
By using Theorem~\ref{t:bisim-not-borel} we conclude
\begin{corollary}
Bisimilarity on $\nlmp F$ is $\Sig_1^1$-complete.
\end{corollary}
\section{Conclusion}
Nondeterministic labelled Markov process combine probabilistic
behavior with internal nondeterminism, over uncountable state
spaces. In this framework, we
considered the problem of describing bisimilarity by using a modal
logic. 

We reviewed the different available notions of `equivalence of
behavior'. They proceed, in some way, from analogous concepts for
LMP. So the problem of \emph{logical characterization of bisimilarity}
is manifold, depending on which notion of bisimulation one is
concerned. In \cite{D'Argenio:2012:BNL:2139682.2139685} it was
established that the three concepts of bisimilarity (traditional,
state, and event) are indeed
different. The counterexamples 
were image-uncountable process, i.e., having an uncountable number of
probabilistic behaviors for each pair 
$\langle$state,action$\rangle$. For only one of these concepts (event
bisimilarity) did the logical characterization go through, in a
similar way it was done for LMP in \cite{coco}.

In the other extreme of the spectrum, the case of image-finite NLMP
was completely solved in \cite{D'Argenio:2012:BNL:2139682.2139685},
where all bisimilarities coincide and they are characterized by a neat
modal logic.

The case left was that of image-denumerable processes. In this
restricted setting, we  have that traditional and state bisimilarity
coincide and the logic characterizing the event based has uncountably
many formulas. But in all approaches to logical characterization, a
countable measurable logic is needed.

We restricted the image-denumerable case further by considering only
nonprobabilistic NLMP, and its equivalent formulation as  MLTS. With
this simplification, we are led essentially to consider plain  LTS 
and Milner's bisimilarity since this is the same as the other
notions on image-countable MLTS. We showed that for the latter family
of processes,  there is no countable measurable
logic characterizing bisimilarity. We did this in a very strong sense,
by pointing out a specific NLMP $\nlmp F$ whose base space $F$ is
Polish such that 
bisimilarity on $\nlmp F$ is an
analytic non Borel subset  of $F^2$;  therefore there is no countable measurable logic
that characterizes bisimilarity for this process, and \textit{a
  fortiori}, for any class of processes containing $\nlmp F$ (v.g. the
class of all image-infinite NLMP with a Polish state-space). As an
intermediate technical step, we proved that the relation of
bisimilarity on the Polish space of all trees on $\N$ is
analytic-complete. Then the space  $F$ is essentially this Polish space.


After obtaining these results and recalling the use of MLTS in other works
(the counterexamples in
\cite{D'Argenio:2012:BNL:2139682.2139685,Wolovick} are
non-probabilistic NLMP), we conclude that
these models provide a simple 
framework that can be considered as a first test scenario for conjectures
about nondeterministic and probabilistic processes over continuous
state spaces.
\begin{ack}
I want to thank Prof.~Xavier Caicedo for a nice discussion concerning
modal logics and for his Example~\ref{ex:caicedo}. Also C.~Areces
and M.~Campercholi pointed out several necessary references. I acknowledge the
careful reading by both of the referees, and their comments, that changed
drastically the way the results were presented; I realized that some
of the passages in my original version were truly difficult to read,
so thanks again. I would also like 
to thank Jos\'e G.~Mijares for some comments on 
descriptive-set-theoretical issues. Finally, I would like to express
my gratitude to Prof.~Pedro D'Argenio for a very instructive talk on
process algebra.
\end{ack}


\begin{thebibliography}{24}
\expandafter\ifx\csname natexlab\endcsname\relax\def\natexlab#1{#1}\fi
\providecommand{\bibinfo}[2]{#2}
\ifx\xfnm\relax \def\xfnm[#1]{\unskip,\space#1}\fi
\bibitem[{Blackburn et~al.(2006)Blackburn, Benthem and
  Wolter}]{Blackburn:2006:HML:1207696}
\textsc{\bibinfo{author}{P.~Blackburn}, \bibinfo{author}{J.v. Benthem},
  \bibinfo{author}{F.~Wolter}}, ``\bibinfo{title}{Handbook of Modal Logic}'',
  \bibinfo{series}{Studies in Logic and Practical
  Reasoning}~\textbf{\bibinfo{volume}{3}}, \bibinfo{publisher}{Elsevier Science
  Inc.}, \bibinfo{address}{New York, NY, USA} (\bibinfo{year}{2006}).
\bibitem[{Bolognesi and Brinksma(1987)}]{DBLP:journals/cn/BolognesiB87}
\textsc{\bibinfo{author}{T.~Bolognesi}, \bibinfo{author}{E.~Brinksma}},
  \bibinfo{title}{Introduction to the {ISO} specification language {LOTOS}},
  \textit{\bibinfo{journal}{Computer Networks}} \textbf{\bibinfo{volume}{14}}:
  \bibinfo{pages}{25--59} (\bibinfo{year}{1987}).
\bibitem[{Celayes(2006)}]{Celayes}
\textsc{\bibinfo{author}{P.~Celayes}}, ``\bibinfo{title}{Procesos de {M}arkov
  Etiquetados sobre Espacios de {B}orel Est\'andar}'', Master's thesis, FaMAF,
  Universidad Nacional de C\'ordoba (\bibinfo{year}{2006}).
\bibitem[{Danos et~al.(2006)Danos, Desharnais, Laviolette and
  Panangaden}]{coco}
\textsc{\bibinfo{author}{V.~Danos}, \bibinfo{author}{J.~Desharnais},
  \bibinfo{author}{F.~Laviolette}, \bibinfo{author}{P.~Panangaden}},
  \bibinfo{title}{Bisimulation and cocongruence for probabilistic systems},
  \textit{\bibinfo{journal}{Inf. Comput.}} \textbf{\bibinfo{volume}{204}}:
  \bibinfo{pages}{503--523} (\bibinfo{year}{2006}).
\bibitem[{D'Argenio et~al.(2009)D'Argenio, Wolovick, S{\'a}nchez~Terraf and
  Celayes}]{DWTC09:qest}
\textsc{\bibinfo{author}{P.~D'Argenio}, \bibinfo{author}{N.~Wolovick},
  \bibinfo{author}{P.~S{\'a}nchez~Terraf}, \bibinfo{author}{P.~Celayes}},
  \bibinfo{title}{Nondeterministic labeled {M}arkov processes: Bisimulations
  and logical characterization}, in: \bibinfo{booktitle}{QEST},
  \bibinfo{publisher}{IEEE Computer Society}: \bibinfo{pages}{11--20}
  (\bibinfo{year}{2009}).
\bibitem[{D'Argenio et~al.(2012)D'Argenio, S\'{a}nchez~Terraf and
  Wolovick}]{D'Argenio:2012:BNL:2139682.2139685}
\textsc{\bibinfo{author}{P.R. D'Argenio},
  \bibinfo{author}{P.~S\'{a}nchez~Terraf}, \bibinfo{author}{N.~Wolovick}},
  \bibinfo{title}{Bisimulations for non-deterministic labelled {M}arkov
  processes}, \textit{\bibinfo{journal}{Mathematical Structures in Comp. Sci.}}
  \textbf{\bibinfo{volume}{22}}: \bibinfo{pages}{43--68}
  (\bibinfo{year}{2012}).
\bibitem[{Desharnais(1999)}]{Desharnais}
\textsc{\bibinfo{author}{J.~Desharnais}}, ``\bibinfo{title}{Labeled {M}arkov
  Process}'', Ph.D. thesis, McGill University (\bibinfo{year}{1999}).
\bibitem[{Desharnais et~al.(2002)Desharnais, Edalat and Panangaden}]{DEP}
\textsc{\bibinfo{author}{J.~Desharnais}, \bibinfo{author}{A.~Edalat},
  \bibinfo{author}{P.~Panangaden}}, \bibinfo{title}{Bisimulation for labelled
  {M}arkov processes}, \textit{\bibinfo{journal}{Inf. Comput.}}
  \textbf{\bibinfo{volume}{179}}: \bibinfo{pages}{163--193}
  (\bibinfo{year}{2002}).
\bibitem[{Desharnais et~al.(2011)Desharnais, Laviolette and
  Turgeon}]{DBLP:journals/iandc/DesharnaisLT11}
\textsc{\bibinfo{author}{J.~Desharnais}, \bibinfo{author}{F.~Laviolette},
  \bibinfo{author}{A.~Turgeon}}, \bibinfo{title}{A logical duality for
  underspecified probabilistic systems}, \textit{\bibinfo{journal}{Inf.
  Comput.}} \textbf{\bibinfo{volume}{209}}: \bibinfo{pages}{850--871}
  (\bibinfo{year}{2011}).
\bibitem[{Doberkat(2005)}]{DBLP:journals/siamcomp/Doberkat05}
\textsc{\bibinfo{author}{E.E. Doberkat}}, \bibinfo{title}{Stochastic relations:
  Congruences, bisimulations and the hennessy--milner theorem},
  \textit{\bibinfo{journal}{SIAM J. Comput.}} \textbf{\bibinfo{volume}{35}}:
  \bibinfo{pages}{590--626} (\bibinfo{year}{2005}).
\bibitem[{Doberkat(2007{\natexlab{a}})}]{Doberkat2007638}
\textsc{\bibinfo{author}{E.E. Doberkat}}, \bibinfo{title}{{K}leisli morphisms
  and randomized congruences for the {G}iry monad},
  \textit{\bibinfo{journal}{Journal of Pure and Applied Algebra}}
  \textbf{\bibinfo{volume}{211}}: \bibinfo{pages}{638--664}
  (\bibinfo{year}{2007}).
\bibitem[{Doberkat(2007{\natexlab{b}})}]{doberkat2007stochastic}
\textsc{\bibinfo{author}{E.E. Doberkat}}, ``\bibinfo{title}{Stochastic
  Relations: Foundations for {M}arkov Transition Systems}'', Chapman \&
  Hall/CRC Studies in Informatics Series, \bibinfo{publisher}{Taylor \&
  Francis} (\bibinfo{year}{2007}).
\bibitem[{Doberkat(2009)}]{doberkat2009stochastic}
\textsc{\bibinfo{author}{E.E. Doberkat}}, ``\bibinfo{title}{Stochastic
  Coalgebraic Logic}'', Monographs in theoretical computer science,
  \bibinfo{publisher}{Springer} (\bibinfo{year}{2009}).
\bibitem[{Edalat(1999)}]{Edalat}
\textsc{\bibinfo{author}{A.~Edalat}}, \bibinfo{title}{Semi-pullbacks and
  bisimulation in categories of {M}arkov processes},
  \textit{\bibinfo{journal}{Mathematical Structures in Comp. Sci.}}
  \textbf{\bibinfo{volume}{9}}: \bibinfo{pages}{523--543}
  (\bibinfo{year}{1999}).
\bibitem[{Kechris(1994)}]{Kechris}
\textsc{\bibinfo{author}{A.S. Kechris}}, ``\bibinfo{title}{Classical
  Descriptive Set Theory}'', \bibinfo{series}{Graduate Texts in Mathematics}
  \textbf{\bibinfo{volume}{156}}, \bibinfo{publisher}{Springer-Verlag}
  (\bibinfo{year}{1994}).
\bibitem[{Kechris(2011)}]{Kechris-errata}
\textsc{\bibinfo{author}{A.S. Kechris}}, \bibinfo{title}{{C}lassical
  {D}escriptive {S}et {T}heory; corrections and updates},
  \bibinfo{howpublished}{Webpage},  (\bibinfo{year}{2011}).
  \bibinfo{note}{{\texttt{http://www.math.caltech.edu/papers/CDST-corrections.%
pdf}}}.
\bibitem[{Kracht(1999)}]{kracht1999tools}
\textsc{\bibinfo{author}{M.~Kracht}}, ``\bibinfo{title}{Tools and techniques in
  modal logic}'', Studies in logic and the foundations of mathematics,
  \bibinfo{publisher}{Elsevier} (\bibinfo{year}{1999}).
\bibitem[{Larsen and Skou(1991)}]{LarsenSkou}
\textsc{\bibinfo{author}{K.G. Larsen}, \bibinfo{author}{A.~Skou}},
  \bibinfo{title}{Bisimulation through probabilistic testing},
  \textit{\bibinfo{journal}{Inf. Comput.}} \textbf{\bibinfo{volume}{94}}:
  \bibinfo{pages}{1--28} (\bibinfo{year}{1991}).
\bibitem[{Moschovakis(2009)}]{moschovakisDesc}
\textsc{\bibinfo{author}{Y.N. Moschovakis}}, ``\bibinfo{title}{Descriptive Set
  Theory}'', Mathematical Surveys and Monographs, \bibinfo{publisher}{American
  Mathematical Society} (\bibinfo{year}{2009}), \bibinfo{edition}{2} edition.
\bibitem[{Rutten(2000)}]{Rutten00}
\textsc{\bibinfo{author}{J.J.M.M. Rutten}}, \bibinfo{title}{Universal
  coalgebra: a theory of systems}, \textit{\bibinfo{journal}{Theor. Comput.
  Sci.}} \textbf{\bibinfo{volume}{249}}: \bibinfo{pages}{3--80}
  (\bibinfo{year}{2000}).
\bibitem[{S{\'a}nchez~Terraf(2011)}]{Pedro20111048}
\textsc{\bibinfo{author}{P.~S{\'a}nchez~Terraf}}, \bibinfo{title}{Unprovability
  of the logical characterization of bisimulation},
  \textit{\bibinfo{journal}{Information and Computation}}
  \textbf{\bibinfo{volume}{209}}: \bibinfo{pages}{1048--1056}
  (\bibinfo{year}{2011}).
\bibitem[{Scott(1964)}]{Scott64}
\textsc{\bibinfo{author}{D.~Scott}}, \bibinfo{title}{Invariant {B}orel sets},
  \textit{\bibinfo{journal}{Fund. Math.}} \textbf{\bibinfo{volume}{56}}:
  \bibinfo{pages}{117--128} (\bibinfo{year}{1964}).
\bibitem[{Srivastava(2008)}]{books/daglib/0029964}
\textsc{\bibinfo{author}{S.M. Srivastava}}, ``\bibinfo{title}{A Course on
  {B}orel Sets}'', \bibinfo{series}{Graduate texts in mathematics}
  \textbf{\bibinfo{volume}{180}}, \bibinfo{publisher}{Springer}
  (\bibinfo{year}{2008}).
\bibitem[{Wolovick(2012)}]{Wolovick}
\textsc{\bibinfo{author}{N.~Wolovick}}, ``\bibinfo{title}{Continuous
  Probability and Nondeterminism in Labeled Transition Systems}'', Ph.D.
  thesis, Universidad Nacional de C\'ordoba (\bibinfo{year}{2012}).

\end{thebibliography}
\providecommand{\noopsort}[1]{}
\begin{small}\end{small}

\bigskip

\begin{small}
\begin{quote}
CIEM --- Facultad de Matem\'atica, Astronom\'{\i}a y F\'{\i}sica 
(Fa.M.A.F.) 

Universidad Nacional de C\'ordoba --- Ciudad Universitaria

C\'ordoba 5000. Argentina.
\end{quote}
and
\begin{quote}
Algebraische und Logische Grundlagen der Informatik

Institut f\"ur
Theoretische Informatik

Technische Universit\"at Dresden --- Fakult\"at Informatik 

01062 Dresden
\end{quote}
\end{small}
\end{document}